\documentclass[a4paper,11pt]{amsart}
\usepackage{amstext,amsfonts,amsthm,graphicx,amssymb,amscd,epsfig}
\usepackage{amsmath}
\usepackage[ansinew]{inputenc}    
\usepackage{psfrag}
\usepackage{mathrsfs}
\usepackage{xypic}
\usepackage{graphicx}
\usepackage[all,knot,arc,import,poly]{xy}
\usepackage{color}

\theoremstyle{definition}
\newtheorem{definition}{Definition}[section]
\newtheorem{ex}[definition]{Example}
\newtheorem{rem}[definition]{Remark}

\theoremstyle{plain}
\newtheorem{prop}[definition]{Proposition}

\newtheorem{coro}[definition]{Corollary}
\newtheorem{teo}[definition]{Theorem}

\newfont{\bbb}{msbm10 scaled\magstephalf}     

\def\R{\mathbb R}

\def\I{\'{\i}}

\def\R{\mbox{\bbb R}}

\def\A{\mathscr A}
\def\M{\mathcal M}

\def\S{\mathbb S}

\def\reg{\operatorname{reg}}
\def\sing{\operatorname{sing}}
\def\Hess{\operatorname{Hess}}
\def\rank{\operatorname{rank}}
\def\ker{\operatorname{ker}}

\def\I{\operatorname{I}}
\def\II{\operatorname{II}}
\def\III{\operatorname{III}}
\def\VII{\operatorname{VII}}

\title{Geometry of surfaces in $\mathbb R^5$ through projections and normal sections}

\author{J. L. Deolindo Silva, R. Oset Sinha}

\date{}

\address{Departamento de Matem\'atica.
Universidade Federal de Santa Catarina - UFSC, 89036-004 -
Blumenau-SC,  Brazil}

\email{jorge.deolindo@ufsc.br}

\address{Departament de Matem\`{a}tiques,
Universitat de Val\`encia, Campus de Burjassot, 46100 Burjassot,
Spain}

\email{raul.oset@uv.es}

\thanks{Work of J. L. Deolindo Silva partially supported by CAPES/JSPS Grant number 88887.357189/2019--00.}
\thanks{Work of R. Oset Sinha partially supported by MICINN Grant PGC2018-094889-B-I00}


\begin{document}

\begin{abstract}
We study the geometry of surfaces in $\R^5$ by relating it to the geometry of regular and singular surfaces in $\R^4$ obtained by orthogonal projections. In particular, we obtain relations between asymptotic directions, which are not second order geometry for surfaces in $\R^5$ but are in $\R^4$. We also relate the umbilic curvatures of each type of surface and their contact with spheres. We then consider the surfaces as normal sections of 3-manifolds in $\R^6$ and again relate asymptotic directions and contact with spheres by defining an appropriate umbilic curvature for 3-manifolds.
\end{abstract}

\maketitle

\section{Introduction}

The study of the differential geometry of manifolds in Euclidean spaces of dimension greater than 3 is a classical subject. In his seminal paper \cite{Little}, Little introduced certain objects to study immersed manifolds such as the curvature ellipse for surfaces in $\R^4$. However, it's been in the last 20 years when the introduction of Singularity Theory techniques has given a tremendous impulse to the subject (see \cite{Binotto/Costa/Fuster,BruceTari,GarciaMochidaFusterRuas,MochidaFusterRuas3,MochidaFusterRuas,MochidaFusterRuas2}, for example). As the dimension and codimension of the immersed manifolds grow, deeper singularity theory concepts are needed. Besides this, the attention has recently changed to singular manifolds $M^{k}_{\sing}\subset\mathbb{R}^{n}$, $n>k\geq2$, (\cite{Benedini/Sinha/Ruas,BenediniRuasSacramento,MartinsBallesteros,SUY}) and the relation of their geometry with regular manifolds $M^{k}_{\reg}\subset\mathbb{R}^{n}$, $n>k\geq2$ (\cite{BenediniOset,BenediniOset2,BruceNogueira,BallesterosTari,Oset/Saji,OsetSinhaTari}).

Our main interest in this paper is the geometry of regular surfaces in $\R^5$, i.e. $M^{2}_{\reg}\subset\mathbb{R}^{5}$. This has been studied in \cite{Costa/Fuster/Moraes}, \cite{MochidaFusterRuas3} and \cite{Fuster/Ruas/Tari}, however, our motivation is slightly different. In \cite{BenediniOset2} a commutative diagram between regular and singular surfaces and 3-manifolds through projections and normal sections was established. This diagram induces a commutative diagram between the corresponding curvature loci. The curvature locus is the image in the normal space by the second fundamental form of the unit vectors in the tangent space and it contains all the second order geometry of the manifold at a point. This means that the second order geometry of all these manifolds is related. For example, in \cite{BenediniOset2} a relation between the asymptotic directions at a point $p\in M^n_{\reg}\subset\R^{2n}$ and at $\pi_u(p)\in M^n_{\sing}\subset\R^{2n-1}$, where $\pi_u$ is the orthogonal projection along the tangent direction $u$, is given. Namely, since the second fundamental form and the height functions coincide, a direction is asymptotic for $p$ if and only if it is asymptotic for $\pi_u(p)$. In fact, for $M^2_{\reg}\subset\R^4$ projected to $M^2_{\sing}\subset\R^3$ it is shown in \cite{BenediniOset,OsetSinhaTari} that a point $p$ is hyperbolic/elliptic/parabolic if and only if $\pi_u(p)$ is hyperbolic/elliptic/parabolic. However, for $M^2_{\reg}\subset\mathbb R^5$, the definition of asymptotic directions does not depend only on the second order geometry, and for $M^2_{\sing}\subset\mathbb R^4$ a degenerate normal direction is binormal only if it is in $E_p$, the vector space parallel to the affine space which contains the curvature locus. It is therefore not clear how to relate asymptotic directions for this situation.

In the present study we show how a point $p\in M^2_{\reg}\subset\R^5$ where the second fundamental form has rank 2 behaves like a point in a surface in $\R^4$. More precisely, there are distinguished asymptotic directions at $p$ which coincide and have the same properties as the asymptotic directions at $\pi_u(p)$, both when $u$ is a tangent direction (and the projection is a singular manifold) or a normal direction (and the projection is regular). We also establish relations between the umbilic curvatures of the surfaces in $\R^5$ and of their projected surfaces. We study the contact with spheres of $M^2_{\sing}\subset\R^4$ and show that there exists a unique umbilic focal hypersphere at $p$ if and only if there exists a unique umbilic focal hypersphere at $\pi_u(p)$.

Surfaces in $\R^5$ can also be obtained as normal sections of 3-manifolds in $\R^6$, but how to relate the geometry has not been clear until now. As for the projections, we relate asymptotic directions at a point in the 3-manifold with asymptotic directions at the corresponding point in the normal section. We define an appropriate umbilic curvature for 3-manifolds, study the contact with spheres using this invariant and relate it to the contact with spheres of the surface in $\R^5$ obtained as a normal section.

In Section \ref{prelim} we establish the notation and preliminary constructions. In Section 3 we deal with $M^2_{\reg}\subset\R^5$ projected orthogonally along a tangent direction to $M^2_{\sing}\subset\R^4$. In Section 4 we study regular projections. Section 5 establishes relations between the contact with spheres of $M^2_{\reg}\subset\R^5$ and $M^2_{\sing}\subset\R^4$. Finally, Section 6 deals with $M^2_{\reg}\subset\R^5$ obtained as a normal section of $M^3_{\reg}\subset\R^6$.

{\bf Aknowledgements:} The authors thank their families for understanding, since this work was developed mostly during confinement. The authors also thank Farid Tari for useful conversations. The first author would like to express his gratitude to the Universitat de València, where this work was partially carried out, for its hospitality.


\section{The geometry of regular surfaces in $\mathbb R^4$ and $\mathbb R^5$, regular 3-manifolds in $\mathbb R^6$, and singular surfaces in $\mathbb R^4$}\label{prelim}

In this section, we present some of the notation and basic aspects of regular surfaces in $\R^n$ ($n = 4, 5$), regular 3-manifolds in $\R^6$, and corank 1 surfaces in $\R^4$.  These subjects have been studied extensively but we mainly follow the approach in \cite{Benedini/Sinha/Ruas,Binotto/Costa/Fuster,Costa/Fuster/Moraes,dreibelbis,Livro,Little,MochidaFusterRuas3}.

\subsection{Regular $k$-dimensional manifold in $\mathbb R^n$}\label{regular}

Given a smooth surface $M^{k}_{\reg}\subset\mathbb{R}^{n}$, $n>k\geq2$ and $f:U\rightarrow\mathbb{R}^{n}$
a local parametrisation of $M^{k}_{\reg}$ with $U\subset\mathbb{R}^{2}$ an open subset, let
$\{e_{1},\ldots,e_{n}\}$ be an orthonormal frame of $\mathbb{R}^{n}$ such that at any $u\in U$,
$\{e_{1}(u),\ldots,e_{k}(u)\}$ is a basis for $T_{p}M^{k}_{\reg}$ and $\{e_{k+1}(u),\ldots,e_{n}(u)\}$ is a basis for
$N_{p}M^{k}_{\reg}$ at $p=f(u)$.

The second fundamental form of $M^{k}_{\reg}$ at a point $p$ is defined by
$II_{p}:T_{p}M^{k}_{\reg}\times T_pM^{k}_{\reg}\rightarrow N_{p}M^{k}_{\reg}$ given by
$II_p(u,w)=\pi_2(d^2f(u,w))$,
where $\pi_2:T_p\mathbb{R}^{n}\rightarrow N_pM^{k}_{\reg}$
is the canonical projection on the normal space. 
The {\it curvature locus} is the image of the unit vectors in $T_p M^k_{\reg}$ given by the
map $\eta:\S^{k-1}\subset T_p M^k_{\reg}\to N_pM^k_{\reg}$, where $\eta(u)=II_p (u,u)$. The second fundamental form of $M^k_{\reg}$ at $p$ along a normal vector
field $\nu$ is the bilinear map $II_p^\nu :T_pM^k \times T_pM^k\to \R$ given by $II_p^{\nu}(v,w)=\langle\nu, II_p(v, w)\rangle$.

\subsubsection{Case $k=2$:} Taking $w=w_{1}e_{1}+w_{2}e_{2}\in T_pM^{2}_{\reg}$ and $(x,y)$ local coordinates in $\R^2$, $l_{i}=\langle f_{xx},e_{2+i}\rangle,\ m_{i}=\langle f_{xy},e_{2+i}\rangle$
and $n_{i}=\langle f_{yy},e_{2+i}\rangle$, for $i=1,\ldots,n-2$, are called the coefficients of the second fundamental form. $II_p$ is a quadratic form which can be written as
$$
II_{p}(w,w)=\sum_{i=1}^{n-2}(l_{i}w_{1}^{2}+2m_{i}w_{1}w_{2}+n_{i}w_{2}^{2})e_{2+i}.
$$
The matrix of the second fundamental form with respect to the orthonormal frame above is given by
$$
\alpha_p=\left(
         \begin{array}{ccc}
           l_{1} & m_{1} & n_{1} \\
           \vdots & \vdots & \vdots \\
           l_{n-2} & m_{n-2} & n_{n-2}
         \end{array}
       \right).
$$

If the unit circle in $T_pM^2_{\reg}$ is parametrized by $\theta\in[0,2\pi]$, the curvature locus forms an ellipse in the normal space $N_{p}M^{2}_{\reg}$ given by the image of the map
\begin{equation}
\eta(\theta)=\sum_{i=1}^{n-2}(l_{i}\cos(\theta)^2+2m_{i}\cos(\theta)\sin(\theta)+n_{i}\sin(\theta)^2)e_{2+i}.
\end{equation}\label{ellipse}
It is denoted by $\Delta^n_e$ and called the \emph{curvature ellipse}. In particular, if $u=\cos(\theta) e_{1}+\sin(\theta) e_{2}$ belongs to unit sphere $S^{1}\subset\mathbb R^2$, we get
$II_{p}(u,u)=\eta(\theta)$.

\

\hspace{-0.4cm}{\it Case $n=4$:} The points on a surface $M^2_{\reg}\subset\R^4$ are classified according to the position of the point $p$ with
respect to the curvature ellipse $\Delta_e^4$ ($N_pM^2_{\reg}$ is viewed as an affine plane through $p$). The point $p$ is
called \emph{elliptic/parabolic/hyperbolic} if it is inside/on/outside the curvature ellipse at $p$, respectively. Moreover, $p$ is an {\it inflection point} if the curvature ellipse is a degenerate ellipse, i.e. $\rank\alpha_p=1$ (\cite{Little}). The directions along $\theta$ such that the vector $\eta(\theta)$ is tangent to the curvature ellipse $\Delta_e^4$ are called  {\it asymptotic directions} of $M_{\reg}^2$ in $\R^4$ at $p$.  There are 2/1/0/infinite asymptotic directions at $p$ depending on $p$ being a hyperbolic/parabolic/elliptic/inflection point.

\

\hspace{-0.4cm}{\it Case $n=5$:}  If $Aff_p$ denotes the affine space which contains the curvature ellipse, $E_p$ is the vector subspace parallel to $Aff_p$ in $N_pM^2_{\reg}$ and the orthogonal complement of $E_p$ in $N_pM$ is denoted by $E_p^{\perp}$.
The distance from $p$ to $Aff_p$ is the non-negative real number $\kappa_u^{\reg}$ called {\it umbilic curvature} (\cite{Costa/Fuster/Moraes}).

Consider the subsets of $M^2_{\reg}$
$$
\M_i=\{p\in M^2_{\reg}\;|\;\rank\alpha_p =i\},\; i\leq3.
$$
Points in these subspaces are called $\M_i$-points and can be characterized in terms of the relative positions of $E_p$ and $Aff_p$: If $p\in \M_3$, $Aff_p$ is a plane that does not contain the origin of $N_pM^2_{\reg}$. If $p\in \M_2$, $Aff_p$ is either a plane through the origin of $N_pM^2_{\reg}$ (and thus coincides with $E_p$), or a line that does not contain the origin of $N_pM^2_{\reg}$. It is shown in \cite{MochidaFusterRuas3} that for generically immersed surface in $\R^5$, $M^2_{\reg} = \M_3\cup \M_2$, with $\M_2$ a regular curve on $M^2_{\reg}$. When $p\in \M_1$, $Aff_p$ is a line through the origin $p$ of $N_pM^2_{\reg}$, that is $Aff_p =E_p$.

A point $p\in M^2_{\reg}$ is called {\it semiumbilic} if the curvature ellipse $\Delta_e^5$ degenerates to a line segment that does not contain $p$. In this sense, points in $\M_2$ are either semiumbilics, or points at which $Aff_p$ passes through the origin $p$ of $N_pM^2_{\reg}$. Moreover, a point $p\in\M_2$ can be classified as $\M_2^e$, $\M_2^h$ or $\M_2^p$ according to whether the origin $p$ of $N_pM^2_{\reg}$ lies inside, outside, or on the curvature ellipse $\Delta_e^5$ at $p$. In particular, semiumbilic points can be considered as points of type $\M_2^h$.

A geometrical characterization of points on $M^2_{\reg}$ using singularity theory is carried out in \cite{MochidaFusterRuas3} via the family of height functions
$h : M^2_{\reg} \times S^4 \to  \R$
 given by $h(p,\nu)=\langle p,\nu\rangle$, where $S^4$ denotes the unit sphere in ${\mathbb R}^5$.
  For $\nu$ fixed, the height function $h_\nu(p)=h(p,\nu)$ is singular if and only if $\nu\in N_pM^2_{\reg}$.
It is shown in  \cite{MochidaFusterRuas3} that for a generic surface, $p\in \M_3$ if and only if $h_\nu$ has only $A_k$-singularities for any $\nu\in N_pM^2_{\reg}$. A point $p\in \M_2^h\cup \M_2^e$ (respectively, $p\in \M_2^p$) if and only if there exists $\nu\in N_pM$ such that $h_\nu$ has a singularity of type $D_4^{\pm}$ (respectively, $D_5$) at $p$. This direction $\nu$ is called the {\it flat umbilic direction} and it is perpendicular to $E_p$.
 (\cite{MochidaFusterRuas3}, see also Remark \ref{rem1}).

The quadratic form $II^\nu_p$ can be identified with $\Hess(h_\nu)(p)$ at a $\nu\in N_pM^2_{\reg}$, up to smooth local changes of coordinate in $M^2_{\reg}$ (\cite{Binotto/Costa/Fuster,MochidaFusterRuas3,Fuster/Ruas/Tari}). 
A unit vector $\nu \in N_pM^2_{\reg}$ is called a {\it degenerate direction} if $h_\nu$ has a singularity of type $A_2$ or worse at $p$ and it is a {\it binormal direction} if the singularity is of type $A_3$ or worse at $p$. A unit vector in $u\in T_pM^2_{\reg}$ is called a {\it contact direction} associated to a degenerate direction $\nu$ if $u\in\ker(\Hess(h_\nu)(p))$. For $\M_3$-points, an {\it asymptotic direction} at $p$, is a contact direction associated to the binormal direction $\nu$. For $\M_2$-points where $h_{\nu}$ has an $A_{k\geq 3}$-singularity, any contact direction associated to $\nu$ is called asymptotic. However, when $\nu$ is a flat umbilic direction $\ker(\Hess(h_v)(p)) = T_pM^2_{\reg}$, but only certain directions are called asymptotic (see Section 3 for details). 

\subsubsection{Case $k=3$ and $n=6$:}
In this case, the curvature locus is denoted by $\Delta_v$ and is studied in \cite{Binotto/Costa/Fuster}.
Here, a direction $u\in T_pM^3_{\reg}$ is called an {\it asymptotic direction} of $M^3_{\reg}$ at $p$ if there is a non zero vector $\nu\in N_pM^3_{\reg}$ such that $II_\nu(u,v)=\langle II(u,v),\nu\rangle=0$ for all $v\in T_pM^3_{\reg}$. Moreover, in such case, we say that $\nu$ is a {\it binormal direction}. Equivalently, $u$ is asymptotic if and only if there exists $\nu\in N_pM^3_{\reg}$ such that the height function $h_\nu$ has a degenerate singularity and $u\in\ker \Hess h_\nu$ (see \cite{dreibelbis}). More details about $M^3_{\reg}$ can be found in Section \ref{sec6}.

\subsection{Singular surfaces in $\mathbb R^4$}\label{singular}

Let $M^{2}_{\sing}$ be a corank $1$ surface in $\mathbb{R}^{4}$, at $p$. The singular surface $M^{2}_{\sing}$ will be taken as the image of a smooth map $g:\tilde{M}\rightarrow \mathbb{R}^{4}$, where $\tilde{M}$ is a smooth regular surface and $q\in\tilde{M}$ is a corank $1$ point of $g$ such that $g(q)=p$. Also, consider $\phi:U\rightarrow\mathbb{R}^{2}$ a local coordinate system defined in an open neighborhood $U$ of $q$ at $\tilde{M}$. Using this construction, we may consider a local parametrisation $f=g\circ\phi^{-1}$ of $M^{2}_{\sing}$ at $p$. 


The tangent line of $M^{2}_{\sing}$ at $p$, $T_{p}M^{2}_{\sing}$, is given by $\mbox{Im}\ dg_{q}$, where $dg_{q}:T_{q}\tilde{M}\rightarrow T_{p}\mathbb{R}^{4}$ is the differential map of $g$ at $q$. The normal space of $M^{2}_{\sing}$ at $p$, $N_{p}M^{2}_{\sing}$, is the subspace satisfying $T_{p}M^{2}_{\sing}\oplus N_{p}M^{2}_{\sing}=T_{p}\mathbb{R}^{4}$.

The first fundamental form of $M^{2}_{\sing}$ at $p$, $I:T_{q}\tilde{M}\times T_{q}\tilde{M}\rightarrow \mathbb{R}$ is given by
$I(u,v)=\langle dg_{q}(u),dg_{q}(v)\rangle$, $\forall\ u,v\in T_{q}\tilde{M}$ and the second fundamental form of $M^{2}_{\sing}$ at $p$, $II:T_{q}\tilde{M}\times T_{q}\tilde{M}\rightarrow N_{p}M$ in the basis $\{\partial_{x},\partial_{y}\}$ of $T_{q}\tilde{M}$ is given by
$$
\begin{array}{c}
     II(\partial_{x},\partial_{x})=\pi_2(f_{xx}(\phi(q))),\  II(\partial_{x},\partial_{y})=\pi_2(f_{xy}(\phi(q))),\
      II(\partial_{y},\partial_{y})=\pi_2(f_{yy}(\phi(q)))
\end{array}
$$
where  $\pi_2:T_{p}\mathbb{R}^{4}\rightarrow N_{p}M^{2}_{\sing}$ is the orthogonal projection and is extended to the whole space uniquely as a symmetric bilinear map.

Given a normal vector $\nu\in N_{p}M^{2}_{\sing}$, the second fundamental form along $\nu$, $II_{\nu}:T_{q}\tilde{M}\times T_{q}\tilde{M}\rightarrow\mathbb{R}$ is given by $II_{\nu}(u,v)=\langle II(u,v),\nu\rangle$, for all $u,v\in T_{q}\tilde{M}$.
Thus, if $u=\alpha\partial_{x}+\beta\partial_{y}\in T_{q}\tilde{M}$ and
fixing an orthonormal frame $\{\nu_{1},\nu_2,\nu_{3}\}$ of $N_{p}M^{2}_{\sing}$,
$$
\begin{array}{c}\label{eq.2ff}
II(u,u) =\displaystyle\sum_{i=1}^{3}II_{\nu_{i}}(u,u)\nu_{i}=\sum_{i=1}^{3}(\alpha^{2}l_{\nu_{i}}(q)+2\alpha\beta m_{\nu_{i}}(q)+\beta^{2}n_{\nu_{i}}(q))\nu_{i}. \\
\end{array}
$$

Let $C_{q}\subset T_{q}\tilde{M}$ be the subset of unit tangent vectors and let $\eta:C_{q}\rightarrow N_{p}M$ be the map given by $\eta(u)=II(u,u)$. The \emph{curvature parabola} of $M^{2}_{\sing}$ at $p$, denoted by $\Delta_{p}$, is the subset $\eta(C_q)$. The curvature parabola is a plane curve, and it can degenerate into a half-line, a line or even a point.
This special curve plays a similar role as the curvature ellipse $\Delta_e^4$ does for regular surfaces in $\R^4$. In fact, in \cite{BenediniOset} there is a relation between them.


The minimal affine space which contains the curvature parabola is denoted by $Aff_{p}$. The plane denoted by $E_{p}$ is the vector space: parallel to $Aff_{p}$ when $\Delta_{p}$ is a non degenerate parabola, the plane through $p$ that contains $Aff_{p}$ when $\Delta_{p}$ is a non radial half-line or a non radial line and any plane through $p$ that contains $Aff_{p}$ when $\Delta_{p}$ is a radial half-line, a radial line or a point. The non negative real number $\kappa_u^{\sing}(p)=d(p,Aff_p)$ is called {\it umbilic curvature}.

A non-zero direction $u\in T_{q}\tilde{M}$ is called {\it asymptotic} if there is a non-zero vector $\nu\in E_{p}$ such that
$II_{\nu}(u,v)=\langle II(u,v),\nu\rangle=0$ for all $v\in T_{q}\tilde{M}$.
In such case,  $\nu$ is said be a {\it binormal direction}. A point in $M^{2}_{\sing}$ is {\it hyperbolic/elliptic/parabolic/inflection} if it has 2/0/1/infinte asymptotic directions.

The normal directions $\nu\in N_{p}M^{2}_{\sing}$ which are not in the plane $E_p$ but also satisfy the condition $II_{\nu}(u,v)=\langle II(u,v),\nu\rangle=0$, are called {\it degenerate directions}. The subset of degenerate directions in $N_pM^2_{\sing}$
is a cone and the binormal directions are those in the intersection of this cone with $E_p$.

It is possible to take a coordinate system $\phi$ and make rotations in the target in order to obtain
$f(x,y)=g\circ\phi^{-1}(x,y)=(x,f_{2}(x,y),f_{3}(x,y),f_{4}(x,y))$,
where $\frac{\partial f_{i}}{\partial x}(\phi(q))=\frac{\partial f_{i}}{\partial y}(\phi(q))=0$ for $i=2,3,4$. Taking an orthonormal frame $\{\nu_{2},\nu_3,\nu_{4}\}$ of $N_{p}M^2_{\sing}$, the curvature parabola $\Delta_{p}$ can be parametrized by
$$
\eta(y)=\sum_{i=2}^{4}(l_{\nu_{i}}+2m_{\nu_{i}}y+n_{\nu_{i}}y^{2})\nu_{i}.
$$
where $y\in\mathbb{R}$ corresponds to a unit tangent direction $u=(1,y)\in C_{q}$.


A point $p$ belongs to the subspace $\M_i$ if the rank of the second fundamental form at $p$ is $i$, $i = 0, 1, 2, 3.$ Following \cite{Benedini/Sinha/Ruas}, when $\Delta_p$ is a non degenerate parabola, $p\in\M_2$ or $p\in \M_3$ according to
$Aff_p = E_p$ or not, respectively. In particular,  $p\in\M_3$ iff $\kappa^{\sing}_u(p) \neq0$. If $\Delta_p$ is a half-line or a line, $p\in \M_1$ or $p\in\M_2$ depending on $\Delta_p$ being radial or not, respectively. Here, $p\in\M_2$ iff $\kappa^{\sing}_u(p) \neq0$. Finally,  when $\Delta_p$ is a point, $p\in \M_0$ or $p\in\M_1$ according to $\Delta_p$ is $p$ or not, respectively. In this case, $p\in\M_1$ iff $\kappa^{\sing}_u(p) \neq0$.

\section{Projecting surfaces in $\mathbb R^5$ along a tangent direction}

Let $p\in M_{\reg}^2$ be a point in a smooth surface in $\R^5$ and consider $\pi_u:\mathbb R^5\to\mathbb R^4$ the orthogonal projection along a direction $u\in T_pM_{\reg}^2$ to $\mathbb R^4$. The projection $\pi_u(M_{\reg}^2)$ is a singular surface $M_{\sing}^2\subset\mathbb R^4$.

It is natural to ask whether there is a relation between the umbilic curvatures $\kappa_u^{\reg}$ and $\kappa_u^{\sing}$.

\begin{ex}
Consider a regular surface $M_{\reg}^2$ whose 2-jet is given by $(x, y, x^2, 2xy, y^2)$ and $(X,Y,Z)$ the coordinates of $N_pM_{\reg}^2$. Its curvature ellipse is given by $\eta_e(\theta)=2\big(\cos^2(\theta),\sin(\theta)\cos(\theta),\sin^2(\theta))$. Here $Aff_p$ is the plane $X+Z=2$
and, consequently $k^{\reg}_u(p)=2\sqrt{2}$. When projecting the surface along the tangent direction $u=(0,1)$ we obtain a singular surface whose 2-jet of the parametrisation is $\mathscr A^2$-equivalent to $(x, xy, y^2, 0)$ (i.e. equivalent by 2-jets of smooth changes of coordinates in the source and target) and the curvature parabola is given by
$\eta_p(y)=2\big(1,y,y^2)$. Here $Aff_{\pi_u(p)}$ is the plane $X=2$ and $k^{\sing}_u(\pi_u(p))=2$, so $Aff_p\neq Aff_{\pi_u(p)}$ and $\kappa_u^{\reg}\neq\kappa_u^{\sing}$.
\end{ex}

\begin{definition}
A point $q\in M_{\sing}^2\subset\mathbb R^4$ is called semiumbilic if the curvature parabola $\Delta_p$ is a half-line or a line which is not radial.
\end{definition}

\begin{prop}\label{semiumb}
Let $p\in M^{2}_{\reg}\subset\mathbb{R}^{5}$, $u\in T_p M^2_{\reg}$ and $\pi_u$ be the orthogonal projection such that $\pi_u(p)\in M^2_{\sing}\subset\mathbb R^4$. Let $\kappa_u^{\reg}$ and $\kappa_u^{\sing}$ be the corresponding umbilic curvatures.
\begin{itemize}
\item[i)] If $p\in \M_3$, then $\kappa_u^{\reg}\kappa_u^{\sing}\neq 0$.
\item[ii)] If $p\in \M_2$, the following are equivalent
\begin{itemize}
\item[1)] $\kappa_u^{\reg}=\kappa_u^{\sing}=0$,
\item[2)] $\Delta_e^5$ and $\Delta_p$ are non-degenerate,
\item[3)] Both $p$ and $\pi_u(p)$ are not semiumbilic.
\end{itemize}
The following are equivalent too
\begin{itemize}
\item[1')] $\kappa_u^{\reg}\kappa_u^{\sing}\neq 0$
\item[2')] $\Delta_e^5$ and $\Delta_p$ are degenerate (not a point),
\item[3')] Both $p$ and $\pi_u(p)$ are semiumbilic.
\end{itemize}
\end{itemize}
\end{prop}
\begin{proof}
First of all notice that $p\in \M_i$ if and only if $\pi_u(p)\in \M_i$, $i=0,\ldots,3$.

For an $\M_3$-point, $Aff_p$ and $Aff_{\pi_u(p)}$ are planes that do not contain the origin, so $\kappa_u^{\reg}$  and $\kappa_u^{\sing}$ are non zero. Item ii) follows directly from the considerations in Section \ref{prelim} about the types of ellipses/parabolas in $\M_2$-points and the definition of umbilic curvature.
\end{proof}

Asymptotic directions on $M^2_{\reg}\subset\mathbb R^{5}$ can be described via the $\A$-singularities of the projections of $M^2_{\reg}$ to 4-spaces (i.e. under the action of the group $\A$). If $TS^4$ denotes the tangent bundle of the $4$-sphere $S^4$, the family of projections to
$4$-planes is given by
$$
\begin{array}{rcl}
\Pi : M^2_{\reg} \times S^4& \to & TS^4\\
 (p, u)&\mapsto& (u,\pi_u (p))
 \end{array}
$$
where $\pi_u(p)=p-\langle p,u\rangle u$. For a given $u \in S^4$, the map $\pi_u$ can be considered locally as a germ of a smooth map $\R^2, 0\to \R^4, 0$. 
Then the generic $\A$-singularities of $\pi_{u}$ are those that have $\A_e$-codimension less than or equal to 4 (which is the dimension of $S^4$). These are listed in Table 1 (see \cite{Rieger}).

\begin{table}[htp]\label{singPro} \caption{Local singularities of projections of surfaces in $\R^5$ to $4$-spaces.}
\begin{center}
\begin{tabular}{ccc}
\hline
     Type & Normal form  & $\A_e$-codimension\\
\hline
Immersion &  $(x, y, 0, 0)$ & $0$\\
$\I_k$ & $(x,xy,y^2,y^{2k+1})$, $k = 1,2,3,4$ & $k$\\
$\II_2$ & $(x,y^2,y^3,x^ky)$, $k = 2$ & $3$\\
$\III_{2,3}$ & $(x,y^2,y^3\pm x^ky,x^ly)$, $k=2,l=3$ & $4$\\
$\VII_1$ &   $(x,xy,xy^2\pm y^{3k+1},y^3)$, $k=1$ & $4$\\
\hline
\end{tabular}
\end{center}
\end{table}

Following \cite{Fuster/Ruas/Tari}, for a generic surface $M^2_{\reg}\subset\R^5$, a direction $u\in T_pM^2_{\reg}$, with $p\in \M_3$, is  asymptotic if and only if the projection of $M^2_{\reg}$ along $u$ to a transverse $4$-space has an $\A$-singularity of type $\I_2$ or worse. At generic $\M_3$-points, there are at most $5$ and at least $1$ asymptotic directions $u$ where $\pi_u$ has an $\A$-singularity of type $\I_2$, which correspond to $A_{k\geq3}$-singularities of the height functions $h_\nu$ along the associated binormal direction $\nu$. If $p$ is a generic $\M_2$-point, then there at most $3$ and at least $1$ asymptotic directions where $\pi_u$ has an $\A$-singularity of type $\I_2$. These are dual to the flat umbilic normal direction. There are also two special asymptotic directions (respectively, none) where $\pi_u$ has an $\A$-singularity of type $\II_2$ if $p\in \M_2^h$ (respectively, $p\in \M_2^e$), and
one direction where $\pi_u$ has an $\A$-singularity of type $\VII_1$ if $p\in \M_2^p$.  These directions correspond to $A_{k\geq3}$-singularity of the height functions along the binormal direction associated to $u$.

\begin{prop}
For $\M_3$-points $\pi_u(p)$ is a singularity of type $I_k$ (even for non generic surfaces). In particular, $u$ is asymptotic if and only if $\pi_u(p)$ is a singularity of type $I_k$, $k>1$.
\end{prop}
\begin{proof}
For generic surfaces this is proved in Theorem 3.6 in \cite{Fuster/Ruas/Tari} by direct computation. If $p$ and $\pi_u(p)$ belong to $\M_3$,  then $\Delta^5_e$ and $\Delta_p$ are a non-degenerate ellipse and parabola, respectively. By Theorem 3.6 in \cite{Benedini/Sinha/Ruas}, the curvature parabola for $\pi_u(p)$ is non-degenerate if and only if $j^2\pi_uf(0)\sim_{\mathscr A^2}(x,y^2,xy,0)$, and the only singularities in Table \ref{singPro} in this 2-jet class are the $I_k$ singularities. The rest follows by Theorem 3.6 in \cite{Fuster/Ruas/Tari}.
\end{proof}

\begin{definition}
For $M^2_{\reg}\subset\R^5$ a direction $u\in T_pM^2_{\reg}$ is called $A_k$-asymptotic if it is a contact direction associated to a binormal direction $\nu$ such that $h_{\nu}$ has an $A_{k\geq 3}$-singularity. If it is associated to a binormal direction such that $h_{\nu}$ has an $D_k$-singularity (the flat umbilic) then it is called $D_k$-asymptotic.
\end{definition}

With this definition, for $p\in \M^h_2$ the $D_k$-asymptotic directions $u$ are those such that $\pi_u(p)$ has an $I_k$-singularity and the $A_k$-asymptotic directions are those such that $\pi_u(p)$ is an $II_2$-singularity.

\begin{rem}\label{rem1}
Following \cite{Costa/Fuster/Moraes}, the cone of degenerate directions $\mathcal C_p$ at a point $p\in \M^h_2$ is two planes which intersect in a line which is precisely $\ker \alpha_p\subseteq E_p^{\perp}$. At a point $p\in \M^p_2$ it is a plane containing $\ker \alpha_p$ and at a point $p\in \M^e_2$ it is exactly $\ker \alpha_p$. By Corollary 6 in \cite{MochidaFusterRuas3} the flat umbilic direction is perpendicular to $E_p$ and lies in $\ker \alpha_p$.  All the degenerate directions in a given plane of $\mathcal C_p$ which are not the flat umbilic have a same contact direction. This is because to be a degenerate direction we need that $h_{\nu}$ has a singularity of type $A_2$ or worse, i.e. $\det\Hess(h_{\nu})=0$, and to be the associated contact direction we need that $u\in\ker(\Hess(h_\nu))$. All those degenerate directions have the same Hessian, but only one will be binormal. This means that the unique contact direction associated to all the non-flat umbilic degenerate directions in a given plane of $\mathcal C_p$ is an $A_k$-asymptotic direction.
\end{rem}

\begin{prop}\label{lemma1}
Let $p\in \M_2$. The unitary tangent vector $u$ is an $A_k$-asymptotic direction if and only if $\eta_e(u)$ is tangent to $\Delta_e^5$.
\end{prop}
\begin{proof}
If $p\in \M_2^h$ (resp. $\M_2^e$ and $ \M_2^p$) then $\Delta_e^5$ is a non-degenerate ellipse with $p$ lying outside the ellipse (resp. inside and on the ellipse), and $Aff_p=E_p$ is a plane passing through the origin that contains the image of the second fundamental form.

Let $S^{1}\subset T_pM^2_{\reg}$ be the unit sphere parametrized by $\theta$ and $u=u(\theta)$. We have $\frac{\partial\eta_e(u)}{\partial\theta}=II(u, u)_{\theta} = 2II(u, u_{\theta})$. Since $\{u,u_\theta\}$ is linearly independent, the tangency happens if and only if $\{II(u, u),2II(u, u_{\theta})\}$ is linearly dependent. This happens if and only if there exists a non zero $w\in T_pM^2_{\reg}$ such that $II(u,w)=0$, i.e. $w\in\ker II(u,\cdot)$. This means that the image of the map $II(u,\cdot)$ lies in a line in $E_p$ and so there exists $\nu\in E_p$ such that $\langle II(u,w),\nu\rangle=0$ for all $w\in T_pM^2_{\reg}$, and this is equivalent to $u\in\ker(\Hess(h_\nu))$ with $\nu\in E_p$.

We have proved that $\eta_e(u)$ is tangent to $\Delta_e^5$ if and only if there exists a degenerate direction $\nu\in E_p$ such that $u$ is its associated contact direction. By the above remark, $u$ is in fact an $A_k$-asymptotic direction and $\nu$ is the intersection of $\mathcal C_p$ with $E_p$.
\end{proof}

\begin{rem}
In this sense, the $A_k$-asymptotic directions at $\M_2$-points play the role of asymptotic directions at points in regular surfaces in $\R^4$ (this will be studied further in the next section). The cone $\mathcal C_p$ intersects the plane $E_p$ in exactly two directions which are orthogonal to the tangent rays through the origin to the curvature ellipse ($K_p$ in \cite{Costa/Fuster/Moraes}). The $A_k$-asymptotic directions are dual to the two directions in the intersection of $\mathcal C_p$ with $E_p$ and, similarly to the case of surfaces in $\mathbb R^4$, by the Proposition \ref{lemma1}, the image of the asymptotic directions by the second fundamental form is tangent to the curvature ellipse, i.e. these images are precisely the rays $K_p$.
\end{rem}

\begin{teo}\label{projgeneric} Suppose $M_{\reg}^2$ is a generic surface, then
\begin{itemize}
\item[i)] $p\in \M_2^h$ if and only if $\pi_u(p)$ is a hyperbolic point,
\item[ii)] $p\in \M_2^e$ if and only if $\pi_u(p)$ is an elliptic point,
\item[iii)] $p\in \M_2^p$ if and only if $\pi_u(p)$ is a parabolic point.
\end{itemize}
Moreover, $v$ is an $A_k$-asymptotic direction at $p$ if and only if $v$ is asymptotic for $\pi_u(p)$. That is, the two (resp. one) $A_k$-asymptotic directions for a point $p\in \M_2^h$ (resp. $p\in \M_2^p$) are exactly the only two (resp. one) asymptotic directions at $\pi_u(p)$.
\end{teo}
\begin{proof}
We consider $M_{\reg}^2\subset\R^5$ parametrized in Monge form, i.e. $f(x,y)=\\(x,y,g_{1}(x,y),g_{2}(x,y),g_{3}(x,y))$,
where $\frac{\partial g_{i}}{\partial x}(p)=\frac{\partial g_{i}}{\partial y}(p)=0$ for $i=1,2,3$. By rotation in the normal space generated by $[\nu_3,\nu_4,\nu_5]$ we can take $Aff_p$ to be the plane $\nu_5=0$. The 2-jet of the parametrization now has the form $(0,0,\frac{1}{2}(a_{20}x^2+2a_{11}xy+a_{02}y^2),\frac{1}{2}(b_{20}x^2+2b_{11}xy+b_{02}y^2),0)$. The curvature ellipse $\Delta_e^5$ can be parametrized by
\begin{equation}\label{EllipseCurv}
\begin{array}{rcl}
\eta_e(\theta)&=&\big(a_{20}\cos(\theta)^2+2a_{11}\sin(\theta)\cos(\theta)+a_{02}\sin(\theta)^2,\\
&&b_{20}\cos(\theta)^2+2b_{11}\sin(\theta)\cos(\theta)+b_{02}\sin(\theta)^2,0 \big).
\end{array}
\end{equation}
Now, by rotation in the tangent space we can take $u$ to be $(0,1)$. These rotations will affect the projection but the topological type and position with respect to the origin of the curvature parabola will remain the same. By \cite[ Proposition 3.8]{BenediniOset}, the curvature parabola of the projection along the tangent direction $(0,1)$ is parametrized by
$$\eta_p(y)=(a_{20}+2a_{11}y+a_{02}y^2,b_{20}+2b_{11}y+b_{02}y^2,0).$$

Suppose first that $u$ is a non-asymptotic direction or a $D_k$-asymptotic direction, i.e. an asymptotic direction such that the projection $\pi_u(p)$ is an $\I_k$ singularity. Then $\Delta_{\pi_u(p)}$ is a non-degenerate parabola.
The point $p$ lies on the curvature ellipse if the resultant of the two quadratic polynomials given by the first two components of $\eta_e(\theta)$ is 0, i.e. they have a common root. Similarly the point lies outside/inside the ellipse if the resultant is $>0/<0$. This resultant is the discriminant used to distinguish hyperbolic/elliptic/parabolic points of regular surfaces in $\mathbb R^4$ (\cite{Little,MochidaFusterRuas}). Notice that the same resultant gives the common roots of the quadratic polynomials in $\eta_p(y)$ and so the point $p$ lies inside/outside/on the ellipse if and only if the point $\pi_u(p)$ lies inside/outside/on the parabola. The result follows.

Suppose now that $u$ is an $A_k$-asymptotic direction. We distinguish into cases: If $p\in \M_2^h$ then, by Proposition 3.9 in \cite{Fuster/Ruas/Tari} there are exactly 2 $A_k$-asymptotic directions $u_i$, $i=1,2$, in which the projection has an $\II_2$ singularity. Therefore $j^2\pi_{u_i}f(0)\sim_{\mathscr A^2}(x,y^2,0,0)$ and by Theorem 3.6 in \cite{Benedini/Sinha/Ruas} the curvature parabola is a half-line.
Since the point is in $\M_2$, this half line is not radial and therefore has 2 asymptotic direction, i.e. the projection is a hyperbolic point.

If $p\in \M_2^e$ there are no $A_k$-asymptotic directions.

Finally, if $p\in \M_2^p$ and $u$ is the only $A_k$-asymptotic direction where the projection is an $\VII_1$ singularity, the 2-jet of the projection is equivalent to $(x,xy,0,0)$ and the curvature parabola is a non-radial line, which means that there is only one asymptotic direction and the projection is a parabolic point.

Now, when projecting along $u$, as seen above, $E_p=E_{\pi_u(p)}$. The degenerate direction $\nu$ dual to the $A_k$-asymptotic direction $v$ which lies in $E_p$ may be binormal or not (depending on whether the singularity of $h_{\nu}$ is of type $A_{k\geq 3}$ or $A_2$). In any case, since the height functions for $p\in M_{\reg}^2$ and for $\pi_u(p)\in M_{\sing}^2$ coincide, $\nu$ is a binormal direction at $\pi_u(p)$. The associated contact directions are therefore asymptotic at $\pi_u(p)$ and coincide with the $A_k$-asymptotic directions at $p$.
\end{proof}


\section{Projecting surfaces in $\mathbb R^5$ along a normal direction to 4-space}

Consider $\nu\in N_pM^2_{\reg}$ and let $S_{\nu}$ be the surface patch obtained by projecting $M^2_{\reg}$ orthogonally along $\nu$. The projection is a regular surface in $\R^4$. We are interested in establishing some relation between the special directions in $M^2_{\reg}$ with $S_{\nu}$. In this sense, since the definition of binormal (or asymptotic) directions in $M^2_{\reg}\subset\mathbb R^5$ depend on higher (than second) order singularities of the height function, we need to consider higher order terms of $M^2_{\reg}$.

Following \cite{Fuster/Ruas/Tari}, we choose local coordinates at $p$ and consider the surface taken in Monge form
\begin{equation}\label{normalform}
(x,y,Q_1 + f^1(x,y),Q_2 + f^2(x,y),Q_3+f^3(x,y)),
\end{equation}
 where $Q=(Q_1,Q_2,Q_3)=(x^2,xy,y^2)$ at $\M_3$ points, $Q=(x^2,y^2,0)$ at $\M_2^h$ points, $Q=(x^2,xy,0)$ at $\M_2^p$ points, $Q=(x^2-y^2,xy,0)$ at $\M_2^e$ points, and
 $$
 f^1(x,y) =\sum_{i+j\geq3} a_{ij}x^iy^j,  \;\;\;  f^2(x,y) =\sum_{i+j\geq3} b_{ij}x^iy^j, \;\;\;  f^3(x,y) =\sum_{i+j\geq3} c_{ij}x^iy^j.
 $$
Consider a unit normal vector $\nu=(\nu_3,\nu_4,\nu_5)$. Let us analyze the projection of points $p\in\M_3\cup\M_2$ along $\nu$:

\

\hspace{-0.4cm}{\it $\M_3$-points case:} In this case, we take the Monge form as in (\ref{normalform}).
Without loss of generality, we assume $\nu=(0,0,1)$ belongs to a unit sphere $S^2\subset\R^3$ and $(Z,W,T)$ the coordinates in  $\mathbb R^3$ of $N_pM^2_{\reg}$.  We parametrize the directions near to $\nu$ by $\eta=(\nu_3,\nu_4,\sqrt{1-\nu_3^2-\nu_4^2})$. Instead of the regular projection to the plane $(\nu_3,\nu_4,\sqrt{1-\nu_3^2-\nu_4^2})^\perp$, we project to the fixed plane $(Z,W)$. The modified family of the orthogonal projections is given by
\[
\begin{array}{cccl}
                 \pi: & (\mathbb{R}^2\times \mathbb{R}^2,0) & \rightarrow & (\mathbb{R}^4,0) \\
                  & ((x,y),(\nu_3,\nu_4)) & \mapsto & \pi_{\eta}(x,y),
                  \end{array}
\]
where
$$
\pi_{\eta}(x,y)=\Big(x,y,x^2+f^1(x,y)-\nu_3(y^2+f^3(x,y)),xy+f^2(x,y)-\nu_4(y^2+f^3(x,y))\Big)
$$
and $\pi_\nu(x,y)=(x,y,x^2+f^1(x,y),xy+f^2(x,y))$
are parametrizations of the surfaces $S_{\eta}$ and $S_\nu$, respectively. Moreover, at generic points, by \cite{Fuster/Ruas/Tari}, the $A_k$-asymptotic directions $u=(u_1,u_2)$ are given as solutions of the following equation
\begin{equation}\label{asymp}
\begin{array}{c}
c_{30}u_1^5 - (2b_{30}-c_{21})u_1^4u_2 + (c_{12}-2b_{21} + a_{30})u_1^3u_2^2+(c_{03}-2b_{12}+ a_{21})u_1^2u_2^3  \\
+ (a_{12}-2b_{03})u_1u_2^4 + a_{03}u_2^5 = 0.
\end{array}
\end{equation}
For any normal direction $\nu$ (not necessarily a binormal direction), we can associate the number of $A_k$-asymptotic directions with the geometry of $S_\eta$ and $S_\nu$ as follows.

\begin{prop} With the above conditions, given $p\in \M_3$, we have the following:
\begin{itemize}
\item[i)] $\pi_{\eta}(p)$ is a hyperbolic point $S_\eta$ if and only if $\nu_4^2-\nu_3<0$.
\item[ii)] $\pi_{\eta}(p)$ is an elliptic point  $S_\eta$ if and only if $\nu_4^2-\nu_3>0$.
\item[iii)] $\pi_{\eta}(p)$ is a parabolic point  $S_\eta$ if and only if $\nu_4^2-\nu_3=0$. There are at most $5$ and at least $1$ normal directions $\eta$ where the surface $S_{\eta}$ has a $P_3(c)$-point at the origin. In addition, the number of the $A_k$-asymptotic directions in $\M_3$-points and the number of the points where the asymptotic direction of $S_{\eta}$ is tangent to its parabolic set are the same. These are the solutions of the following equation
\begin{equation}\label{P3}
\begin{array}{c}
c_{30}\nu_4^5 - (2b_{30}-c_{21})\nu_4^4 + (c_{12}-2b_{21} + a_{30})\nu_4^3 + (c_{03}-2b_{12} + a_{21})\nu_4^2 \\
+ (a_{12}-2b_{03})\nu_4 + a_{03} = 0.
\end{array}
\end{equation}
Moreover, the unique asymptotic direction at $P_3(c)$-point in $S_\nu$ is an $A_k$-asymptotic direction in $M^5_{\reg}$.
\end{itemize}
\end{prop}
\begin{proof}
Consider $(x,y,x^2-\nu_3y^2,xy-\nu_4y^2)$ the 2-jet of $\pi_{\eta}$ at the origin. The resultant of the last two coordinates of $j^2\pi_{\eta}(0)$ is $\Delta=\nu_4^2-\nu_3$, so a point at surface is hyperbolic/elliptic/parabolic if and only if $\Delta<$/$>$/$=0$. The other part of the proof of statement iii) follows by making changes of coordinates in order to reduce the appropriate jet of $\pi_{\eta}$.
When projecting a regular surface in $\R^4$ along the unique asymptotic direction $u$ at a parabolic point, the projection has a $P_3(c)$-singularity at isolated points on the parabolic set,  i.e. it is $\A$-equivalent to $(x,y)\mapsto(x,x^2y+cy^4,xy+y^3)$ with $c\in\mathbb R\setminus{\{0,\frac12,1,\frac32\}}$ (see \cite{Deolindo}). In particular, the unique asymptotic direction is tangent to the parabolic set of $S_\eta$ at a $P_3(c)$-point.
As the discriminant of (\ref{asymp}) coincides with the discriminant of (\ref{P3}), we have the statement.
Moreover, the projection along $\nu=(0,0,1)$ gives a regular surface $S_\nu$ which has a $P_3(c)$-point if and only if $a_{03}=0$. In this case, the unique asymptotic direction $u=(0,1)$ at $\pi_\nu(p)\in S_\nu$ is an $A_k$-asymptotic direction at $p$ in $M_{\reg}^5$.
\end{proof}

\hspace{-0.4cm}{\it $\M_2$-points case:} In \cite[Theorem 4.14]{Benedini/Sinha/Ruas}, the authors associate a regular surface in $\R^4$ to any singular surface in $\R^4$. To do so they consider the regular surface in $\R^5$ that gives the singular surface in $\R^4$ when projected along a tangent direction $u$ and project it to the 4-space given by $T_{p}M^2_{\reg}\oplus\pi_u^{-1}(E_{\pi_u(p)})$. This regular surface is $S_{\nu}$ when $\nu\in E_p^{\perp}$ (see the following diagram).
$$
\xymatrix{
&  & M^2_{\reg}\subset\mathbb{R}^{5}\ar[rd]^-{\pi_\nu}\ar[d]^-{\pi_u} & \\
\mathbb{R}^{2}\ar@/_0.7cm/[rr]^-{f}& \tilde{M}\ar[l]_-{\phi}\ar[r]^-{g}\ar[ru]^-{i}& M^2_{\sing}\subset\mathbb{R}^{4}& S_\nu\subset\mathbb{R}^{4}
}
$$
The authors prove that $\pi_u(p)\in M^2_{\sing}$ is an elliptic/hyperbolic/parabolic if and only
$\pi_{\nu}(p) \in S_{\nu}$ is an elliptic/hyperbolic/parabolic, respectively. Furthermore, a direction $v$ is an asymptotic direction of $M^2_{\sing}$ if and only if it is
also an asymptotic direction of the associated regular surface $S_{\nu}$.

In particular, at $\M_2$-points of $M^2_{\reg}$, when $\nu\in E_p^{\perp}$ it is a flat umbilic direction. Then we have the following.


\begin{prop}\label{A3asymp}
Let $p\in \M_2\subset M^2_{\reg}\subset \R^5$ and consider a direction of projection $\nu$ in $E_p^{\perp}$. A direction $v\in T_p M_{\reg}^2$ is $A_k$-asymptotic at $p$ if and only if $v$ is asymptotic at $\pi_{\nu}(p)$ in $S_{\nu}$. Moreover, the second order geometry of $M^2_{\reg}$ and $S_{\nu}$ is the same.\end{prop}
\begin{proof}
The proof follows from Theorem \ref{projgeneric} and Theorem 4.14 in \cite{Benedini/Sinha/Ruas}. The second order geometry is the same because the second fundamental forms at $p\in M^2_{\reg}$ and at $\pi_{\nu}(p)\in S_\nu$ coincide.
\end{proof}

From Proposition \ref{A3asymp} we obtain as a corollary \cite[Theorem 7]{MochidaFusterRuas3} with some extra information about the asymptotic directions:

\begin{coro}\label{SurfR4} If $p\in \M_2$ and the normal direction $\nu\in E_p^{\perp}$, then
\begin{itemize}
\item[i)] $p \in \M_2^h$ if and only if $\pi_{\nu}(p)$ is a hyperbolic point of $S_{\nu}$.
\item[ii)] $p\in \M_2^e$ if and only if $\pi_{\nu}(p)$ is an elliptic point of  $S_{\nu}$
\item[iii)] $p\in \M_2^p$ if and only if $\pi_{\nu}(p)$ is a parabolic point of $S_{\nu}$.
\end{itemize}
Moreover, the $A_k$-asymptotic directions at $p$ coincide with the asymptotic directions at $\pi_{\nu}(p)$.
\end{coro}

\begin{rem}
The previous two results show further how $\M_2$-points of surfaces in $\R^5$ behave as points of regular surfaces in $\R^4$.
\end{rem}

Now, we project $\M_2$-points along binormal directions $\nu$ that are associated to $A_k$-asymptotic directions at $p$.  We have the following result which characterizes when a binormal direction lies in $\mathcal C_p\cap E_p$.

\begin{prop}
Let $p\in \M_2\subset M^2_{\reg}$ and $\nu$ be a binormal direction which is not a flat umbilic. Then $\nu\in E_p$ if and only if the regular surface $S_\nu$ obtained by projecting $M^2_{\reg}$ along $\nu$ has an inflection point at $\pi_\nu(p)$.
\end{prop}

\begin{proof}
Consider  $\M_2\subset M^2_{\reg}$ in Monge form as in (\ref{normalform}) at $p$.  If $\nu$ is not a flat umbilic, then it is associated to an $A_k$-asymptotic direction. At $\M_2^h$-points,
 the two $A_k$-asymptotic directions are dual to the binormal directions $\nu_1=(c_{03},0,-a_{03})$ and $\nu_2=(0,c_{30},-b_{30})$. At $\M_2^p$-points, the unique $A_k$-asymptotic direction is associated to the dual direction $\nu_1$. There is no $A_k$-asymptotic direction at $\M_2^e$-points.
Consider $S_\nu$ the regular projection along the binormal direction $\nu=(c_{03},0,-a_{03})$ (the case $\nu=(0,c_{30},-b_{30})$ is analogous). The projection for a point in $\M_2^h$ (resp.  $\M_2^p$) is locally the regular surface $S_\nu$ parametrized by
$$
(x,y,-a_{03}(Q_1+f^1(x,y))-c_{03}(Q_3+f^3(x,y)),Q_2+f^2(x,y)).
$$
 Since the discriminant of $S_\nu$ is $\Delta= a_{03}^2>0$  (resp. $=0$) (see \cite{Little}), we get that  given $p\in \M_2^h$ (resp. $\M_2^p$), a point in $S_{\nu}$ is a hyperbolic (resp. parabolic)  if and only if $a_{03}$ is non zero. On the other hand, $a_{03}=0$ if and only if $\pi_\nu(p)$ is an inflection point. In particular, $a_{03}=0$ if and only if  $\nu\in E_p$.
\end{proof}

\section{Contact with spheres of regular surfaces in $\mathbb R^5$ and singular surfaces in $\mathbb R^4$}

In \cite{Costa/Fuster/Moraes}, some aspects of the geometry of regular surfaces in $\R^5$ are given by the family of distance squared functions on $M^2_{\reg}$
$$
\begin{array}{rcl}
D : M^2_{\reg} \times S^4 &\to& \R \\
(p, \nu)&\mapsto& ||p-\nu||^2 = d_\nu(p).\\
\end{array}
$$
A distance squared function $d_\nu$ has a singularity at $p\in M^2_{\reg}$ if and only if $\nu\in N_pM^2_{\reg}$. 
A {\it focal center} at $p\in M^2_{\reg}$ is a point at which $d_\nu$ has a degenerate singularity (i.e., $d_\nu$ has a singularity $\A$-equivalent to $A_2$ or worse). The directions lying in the kernel of the corresponding Hessian quadratic form are said to be {\it spherical contact directions} at $p$. A focal center is said to be an {\it umbilical focus} at $p$ if $d_\nu$ has a singularity of corank $2$ at $p$. The corresponding focal hypersphere is called {\it umbilical focal hypersphere}.
The set of the focal centres of the points of $M^2_{\reg}$ is said to be the {\it focal set} of $M^2_{\reg}$ and  denoted by $\mathcal F$. At $p\in \M_3$, there exists a unique umbilical focus $\displaystyle \nu =p+ \frac{1}{k_u^{\reg}}v_p$ with $v_p$ an unit normal vector in $E_p^{\perp}$. 
If $p\in \M_2$ is not semiumbilic, the umbilical focus lies at infinity and the umbilical focal hypersphere becomes an osculating hyperplane, and in the opposite case, there is a line of umbilical foci at $p$ contained in the vector plane orthogonal to the affine line $E_p$ (\cite{Costa/Fuster/Moraes}).

For $M^2_{\sing}\subset \mathbb R^4$, the definitions of focal set and umbilical focus at a singular point $p$ can be extended in a natural way from those of $M^2_{\reg}\subset \R^5$.
We also call a frame $[\nu_2,\nu_3,\nu_4]$ an adapted frame of $N_pM^2_{\sing}$ if it is a positively oriented orthonormal frame, $\nu_4\in E_p^{\perp}$ ($E_p$ is always a plane for $M^2_{\sing}\subset \mathbb R^4$), $[\nu_2,\nu_3]$ is a positively oriented frame of $E_p$ and: $\nu_2$ is parallel to the direction of $\Delta_p$ when $\Delta_p$ is a line or a half-line; $\nu_3$ is the direction of $\Delta_p$ when $\Delta_p$ is a point different from $p$.

\begin{teo} \label{contactspheresing} Let $M^2_{\sing}\subset \mathbb R^4$ be a surface with a singularity of corank 1 at $p\in M^2_{\sing}$ and let $[\nu_2,\nu_3,\nu_4]$ be an adapted frame of $N_pM^2_{\sing}$. The function $d_{\nu}$ is singular at $p$ if and only if $\nu\in N_pM$. Furthermore, the following possibilities hold:
\begin{itemize}
\item[i)]If $p\in \M_3$ 
then the focal set of $M^2_{\sing}$ at $p$ is a quadric. Moreover, it intersects the set of degenerate directions on a line $\nu=p+r(\nu_2+\nu_4)$, $r\in\mathbb R$. In addition, 
the umbilical focus at $p$ is given by
$$
\nu=p+\frac{1}{\kappa^{\sing}_{u}}\nu_4\in E^{\perp}_p.
$$
\item[ii)] If $p\in \M_2$ is semiumbilic and 
\begin{itemize}
\item[1)] 
$\Delta_p$ is a half-line, then the focal set of $M^2_{\sing}$ at $p$ is the union of two transverse planes intersecting in a line of umbilical foci at $p$ given by $\ell=p+\frac{1}{\kappa^{\sing}_{u}}\nu_3+r\nu_4$, $r\in\mathbb R$. 
\item[2)] $\Delta_p$ is a line, then the focal set of  $M^2_{\sing}$ at $p$ is equal to the set of the degenerate directions ($E^{\perp}_p$). In addition, $\ell$ 
 is a line of umbilical foci at $p$. 
\end{itemize}
Moreover, the focal set of $M^2_{\sing}$ at $p$ intersects the set of degenerate directions when $\nu\in E_p^{\perp}$.
\item[iii)] If $p\in \M_1$ and
\begin{itemize}
\item[1)] $\Delta_p$ is a half-line, then the focal set of $M^2_{\sing}$ at $p$ is the union of two parallel planes in $N_pM^2_{\sing}$, which can be coincident.
\item[2)] $p$ is umbilic non flat, then the focal set of $M^2_{\sing}$ at $p$ is the hyperplane $N_pM^2_{\sing}$. In addition, 
there is a plane of umbilical foci at $p$ given by
$$
\nu=p+t\nu_2+\frac{1}{\kappa^{\sing}_{u}}\nu_3+r\nu_4,\;t,r\in\mathbb R.
$$
\end{itemize}
\end{itemize}
\end{teo}
 \begin{proof} The first assertion that $d_\nu$ is singular if and only if $\nu\in N_pM^2_{\sing}$ is obvious.  In all the cases we assume that $p=0$ and $M^2_{\sing}$ is parametrised by $f(x,y)=(x,f_2(x,y),f_3(x,y),f_4(x,y))$ with $j^2f_i(0,0)=0$, for $i=2,3,4$. After smooth
changes of coordinates in the source and isometries in the target, the surface can be taken with one of the following parametrisations (see \cite{Benedini/Sinha/Ruas}):
\begin{itemize}
\item $(x, xy + p(x, y), b_{20}x^2 + b_{11}xy + b_{02}y^2 + q(x, y), c_{20}x^2 + r(x, y))$ iff $\Delta_p$ is a non-degenerate parabola;
\item $(x, a_{20}x^2+y^2 + p(x, y), b_{20}x^2  + q(x, y),  r(x, y))$ iff $\Delta_p$ is a half-line;
\item $(x, xy + p(x, y), b_{20}x^2  + q(x, y),  r(x, y))$ iff $\Delta_p$ is a line;
\item $(x, p(x, y), b_{20}x^2  + q(x, y),  r(x, y))$ iff $\Delta_p$ is a point,
\end{itemize}
where $a_{ij} , b_{ij} , c_{ij}\in\R$, $b_{02} > 0$ and $p, q, r\in\mathscr M_3^2$. If we denote the coordinates in $\R^4$ by $(v_1,v_2,v_3,v_4)$, in all the four cases we have that $[\nu_2,\nu_3,\nu_4] = [(0,1,0,0),(0,0,1,0),(0,0,0,1)]$ is the adapted frame of $N_pM^2_{\sing}$ and take $\nu=(v_1, v_2, v_3,v_4)$. Thus $d_\nu$ is singular at $p$ if and only if $v_1=0$.

 If $p\in \M_3$ then $\Delta_p$ is a non-degenerate parabola. The degenerate directions are given when the height function $h_\nu$ has singularity $A_2$ or worse, i.e.  are solutions of the equation
 $$
-v_2^2-2b_{11}v_2v_3-2v_2v_4+(4b_{02}b_{20}-b_{11}^2)v_3^2+(4b_{02}c_{20}-2b_{11})v_3v_4-v_4^2=0.
 $$
 Furthermore, $d_v$ is given by
 $$
 d_\nu(x,y)=(v_2-xy)^2 +(v_3-(b_{20}x^2+b_{11}xy+b_{02}y^2))^2 +(v_4-c_{20}x^2)^2 +s(x,y),
 $$
 where $s\in\mathscr M^5_2$.
  The focal set of $M^2_{\sing}$ at $p$ is given by solutions of the equation
 $$
-v_2^2-2b_{11}v_2v_3-2v_2v_4+(4b_{02}b_{20}-b_{11}^2)v_3^2+(4b_{02}c_{20}-2b_{11})v_3v_4-v_4^2-2b_{02}v_3=0
 $$
which is a quadric and it intersects the set of degenerate directions when $v_3=0$.
Moreover $d_v$ has a singularity of corank $2$ if and only if $\kappa_u^{\sing}=2c_{20}\neq0$ and $\nu=\left(0,0,0,\frac{1}{\kappa_u^{\sing}}\right)$.

Suppose that $p\in \M_2$ is semiumbilic. If $\Delta_p$ is a half-line then the degenerate directions are given by
 $v_2(a_{20}v_2+b_{20}v_3)=0$ and
 $$
d_\nu(x,y)=(v_2-(a_{20}x^2+y^2)^2 +(v_3-b_{20}x^2)^2 +v_4^2 +s(x,y),
$$
where $s\in\mathscr M^5_2$. The focal set of $M^2_{\sing}$ at $p$ is given by solutions of the equation
 $v_2(2a_{20}v_2+2b_{20}v_3-1)=0$
that  intersect the set of degenerate directions in the plane $E_p^{\perp}$ ($v_2=0$). In this case, if $\kappa_u^{\sing}=2b_{20}\neq0$ we have two transverse planes, intersecting on the line $(0, 0,\frac{1}{\kappa_u^{\sing}}, v_4)$. Otherwise, if $\kappa_u^{\sing} = 0$, then $p\in \M_1$ we have two parallel planes. In both cases, $v_2 = 0$ is one of the planes. Moreover, there is a corank 2 singularity only in the case that $\kappa_u^{\sing}\neq0$ and $\nu = \left(0, 0,\frac{1}{\kappa_u^{\sing}}, v_4\right)$. Now suppose that $\Delta_p$ is a line. The distance-squared function is given by
$$
d_\nu(x,y)=(v_2-xy)^2 +(v_3-b_{20}x^2)^2 +v_4^2 +s(x,y)
$$
where $s\in\mathscr M^5_2$. The focal set and the set of degenerate directions are equal to $E_p^{\perp}$ ($v_2=0$). Furthermore, the singularity is of corank 2 if and only if $\kappa_u^{\sing}=2b_{20}\neq0$ and $\nu = \left(0, 0,\frac{1}{\kappa_u^{\sing}}, v_4\right)$.

Finally, suppose $p\in\M_1$ is umbilic non-flat. Here $\Delta_p$ is a point different of $p$, then
$$
d_\nu(x,y)=v_2^2 +(v_3-b_{20}x^2)^2 +v_4^2 +s(x,y),
$$
where $s\in\mathscr M^5_2$. The focal set and the set of degenerate directions are equal to $N_pM$. Moreover, $d_\nu$ has a singularity of corank $2$ if and only if  $\kappa_u^{\sing}=2b_{20}\neq0$ and $\nu = \left(0, v_2,\frac{1}{\kappa_u^{\sing}}, v_4\right)$.
\end{proof}

\begin{coro}
Let $p\in M^{2}_{\reg}\subset\mathbb{R}^{5}$, $u\in T_p M^2_{\reg}$ and $\pi_u$ be the orthogonal projection such that $\pi_u(p)\in M^2_{\sing}\subset\mathbb R^4$.
\begin{itemize}
\item[i)] There exists a unique umbilical focus at $p\in M^{2}_{\reg}$ if and only if there exists a unique umbilical focus at $\pi_u(p)\in M^2_{\sing}$.
\item[ii)] The point $p\in \M_2$ is not semiumbilic and the umbilical focus of $M^{2}_{\reg}$  lies at infinity if and only if the point $\pi_u(p)\in\M_2$ is not semiumbilic and the umbilical focus of $M^{2}_{\sing}$ lies at infinity. Moreover, in both cases the umbilical focal hypersphere becomes an osculating hyperplane.
\item[iii)] The point $p\in\M_2$ is semiumbilic and there is a line of umbilical foci at $p$ in  $E_p^{\perp}$ if and only if the point $\pi_u(p)\in\M_2$ is semiumbilic and  there is a line of umbilical foci at $\pi_u(p)$ in $E_p^{\perp}$. Furthermore,  the line is a degenerate direction.
\end{itemize}
\end{coro}
\begin{proof}
The proof follows from Theorem \ref{contactspheresing} and the above comments. In case ii), in both cases there is no proper umbilical focal hypersphere, but the height function in the direction $\nu$ has a corank 2 singularity at $p$ and $\pi_u(p)$, and hence, the hyperplane with orthogonal direction $\nu$ is an osculating hyperplane (with corank 2 contact) that can be considered as a degenerate umbilical focal hypersphere.
\end{proof}


\section{Surfaces in $\mathbb R^5$ as normal sections of 3-manifolds in $\mathbb R^6$}\label{sec6}

Following \cite{BenediniOset2},
if $M^3_{\reg}\subset \mathbb R^6$ is given in Monge form by
$$
(x,y,z,f_1(x,y,z),f_2(x,y,z),f_3(x,y,z)),
$$
let $u=(0,0,1)\in T_p M^3_{\reg}$ and let $\pi_u$ be the projection along the direction $u$. Consider the normal section given by $\{Y=0\}$. Let $i_1, i_2$ be the immersions of the normal sections in $\mathbb R^6$ and $\mathbb R^5$ respectively. Let $u'=i_{1_*}^{-1}(u)=(0,1)\in T _{i_1^{-1}(p)}M^2_{\reg}$. Then the following diagram is commutative
\begin{equation*}
\begin{CD}
M^3_{\reg}\subset\mathbb R^6 @>{\pi_u}>> M^3_{\sing}\subset\mathbb R^5\\ @A{i_1}AA        @AA{i_2}A   \\
M^2_{\reg}\subset\mathbb R^5 @>>{\pi_{u'}}> M^2_{\sing}\subset\mathbb R^4
\end{CD}
\end{equation*}
where $M^2_{\reg}=M^3_{\reg}\cap\{Y=0\}$ and $M^3_{\sing}, M^2_{\sing}$ are the corresponding singular projections, and induces a commutative diagram amongst the curvature loci of the four manifolds.

\begin{teo}
Suppose that $q\in \M_2\subset M_{\reg}^2\in\R^5$ is not semiumbilic and suppose that $u'\in T _{q}M^2_{\reg}$ is not an $A_k$-asymptotic direction.
The following are equivalent:
\begin{itemize}
\item[i)] $v$ is $A_k$-asymptotic at $q$
\item[ii)] $v$ is asymptotic at $\pi_{u'}(q)$
\item[iii)] $i_2^*(v)$ is asymptotic at $i_2(\pi_{u'}(q))=\pi_u(i_1(q))$
\item[iv)] $i_1^*(v)$ is asymptotic at $i_1(q)$
\end{itemize}
\end{teo}
\begin{proof}
Equivalence between i) and ii) is proven in Theorem \ref{projgeneric}. Since $q$ is not semiumbilic the curvature ellipse at $q$ is non degenerate, and since $u'$ is not an $A_k$-asymptotic direction, $\pi_{u'}(q)$ is an $I_k$ singularity. Therefore the curvature parabola at $\pi_{u'}(q)$ is non degenerate and by Proposition \ref{semiumb} $\pi_{u'}(q)$ is not semiumbilic. Thus we have $Aff_{\pi_{u'}(q)}=E_{\pi_{u'}(q)}$ and by Theorem 5.9 in \cite{BenediniOset2} we get equivalence between ii) and iii). Equivalence between i) and iv) follows similarly to Theorem 5.9 in \cite{BenediniOset2} since by the Proposition \ref{lemma1} the image of the asymptotic directions is tangent to the curvature locus. Equivalence between iii) and iv) follows from the fact that the second fundamental forms are the same and the condition for being an asymptotic direction depends on $A_2$-singularities of the height function in the dual direction in both cases.
\end{proof}

\begin{ex}
Consider a 3-manifold whose 2-jet is given by $(x,y,z,x^2+1/2z^2,xz,yz)$ and consider the projection in the tangent direction $(0,0,1)$ and the normal section given by $\{Y=0\}$. $(1,\pm\sqrt{2})$ is asymptotic at $q$ and $\pi_{u'}(q)$ and $(1,0,\pm\sqrt{2})$ is asymptotic at $i_2(\pi_{u'}(q))=\pi_u(i_1(q))$ and at $i_1(q)$.
\end{ex}

For $p\in M^3_{\reg}\subset \mathbb R^6$ the matrix of the second fundamental form with respect to the frame given in Section \ref{prelim} is
$$
\alpha_p=\left(
         \begin{array}{cccccc}
           a_{1} & b_{1} & c_{1} & d_1 & r_1 & s_1 \\
          a_{2} & b_{2} & c_{2} & d_2 & r_2 & s_2 \\
          a_{3} & b_{3} & c_{3} & d_3 & r_3 & s_3
         \end{array}
       \right).
$$
In
\cite{Binotto/Costa/Fuster} it is shown that the curvature locus of $M^{3}_{\reg}$
at $p$ can be parametrised by
$\eta:S^{2}\subset T_pM^{3}_{\reg}\rightarrow N_pM^{3}_{\reg},\ (\theta,\phi)\mapsto\eta(\theta,\phi)$,
where
$$
\begin{array}{cl}
    \eta(\theta,\phi)= & H+(1+3\cos(2\phi))B_1+\cos(2\theta)\sin^2\theta B_2 \\
     & +\sin(2\theta)\sin^2\phi B_3+cos\theta\sin(2\phi)B_4+\sin\theta\sin(2\phi)B_5
\end{array}
$$
with
$$
\begin{array}{c}
     H=\frac{1}{3}\sum_{i=1}^3(a_i+d_i+s_i)e_{i+3},\ B_1=\frac{1}{12}\sum_{i=1}^3(-a_i-d_i+2s_i)e_{i+3},  \\
     B_2=\frac{1}{2}\sum_{i=1}^3(a_i-d_i)e_{i+3},\ B_3=\sum_{i=1}^3(b_i)e_{i+3},\\
      B_4=\sum_{i=1}^3(c_i)e_{i+3},\ B_5=\sum_{i=1}^3(r_i)e_{i+3}.
\end{array}
$$

The first normal space
is $N_p^1M^{3}_{\reg}=\langle H,B_1,B_2,B_3,B_4,B_5\rangle_{(p)}$. The spaces $Aff_p$ and $E_p$ are defined as in all other regular cases. In fact, $E_p=\langle B_1,B_2,B_3,B_4,B_5\rangle_{(p)}$.
In \cite{Binotto/Costa/Fuster}, it is shown that, generically, $\dim (N_p^1M^{3}_{\reg})=3$, and that if $\dim E_p=2$ (i.e. $H\notin E_p$) the curvature locus is a planar region contained in the plane $Aff_p$. This motivates the following.
\begin{definition}
If $\dim E_p=2$ we define the \emph{umbilic curvature} at a point $p\in M^{3}_{\reg}\subset \mathbb R^6$ as $$\kappa_u^6(p)=d(p,Aff_p).$$
\end{definition}
\begin{rem}
The vector $H$ is called the mean curvature vector and lies in $Aff_p$. If we denote by $\nu$ the unit vector orthogonal to $E_p$, then $\kappa_u^6(p)=\langle H,\nu\rangle.$
\end{rem}
From now on we shall denote by $\kappa_u^5$ the umbilic curvature of $M^{2}_{\reg}\subset \mathbb R^5$.
\begin{prop}\label{sameku}
Consider $p\in M^3_{\reg}\subset \mathbb R^6$ and $q\in M^{2}_{\reg}\subset \mathbb R^5$ obtained as a normal section. Suppose $\dim E_p=2$ and that $q$ is not semiumbilic, then $\kappa_u^6=\kappa_u^5$.
\end{prop}
\begin{proof}
In \cite{BenediniOset2} the authors show that the curvature locus of the 3-manifold can be generated by the curvature loci of all the possible normal sections. If $\dim E_p=2$, the curvature locus of the 3-manifold is a planar region contained in the plane $Aff_p$, therefore the curvature locus of any normal section will be a curve contained in the same plane. Since $q$ is not semiumbilic, $\Delta_e^5$ is non-degenerate and so $Aff_q=Aff_p$. The result follows.
\end{proof}

\subsection{Contact with spheres}

As we did for singular surfaces in $\mathbb R^4$, we consider the family of distance squared functions on $M^3_{\reg}$
$$
\begin{array}{rcl}
D : M^3_{\reg} \times \mathbb R^6 &\to& \R \\
(p, a)&\mapsto& ||p-a||^2 = d_a(p).\\
\end{array}
$$
Following \cite{Binotto/Costa/Fuster}, a point in the focal set (points where the distance squared function has a degenerate singularity) where $\Hess d_a$ has corank 3 is called an umbilical focus. For a generic 3-manifold the rank of the second fundamental form is 3 at every point and the umbilical foci lie at isolated points where the curvature locus is of planar type.

\begin{prop}\label{contact3var}
Let $p\in M^3_{\reg}$ be a point in a generic 3-manifold and suppose $\dim E_p=2$ (i.e. the curvature locus is a planar region), then there exists a unique umbilical focus at $p$ given by $$a=p+\frac{1}{\kappa_u^6}v$$ where $v$ is a unit vector in $E_p^{\perp}$.
\end{prop}
\begin{proof}
For $v=v_4e_4+v_5e_5+v_6e_6\in N_pM^3_{\reg}$ and $a=p+\lambda v$ we have
$$
\Hess d_a=\left(
         \begin{array}{ccc}
          1-\lambda\sum_{i=1}^3a_iv_i & \lambda\sum_{i=1}^3b_iv_i & \lambda\sum_{i=1}^3c_iv_i  \\
          \lambda\sum_{i=1}^3b_iv_i & 1-\lambda\sum_{i=1}^3d_iv_i & \lambda\sum_{i=1}^3r_iv_i  \\
          \lambda\sum_{i=1}^3c_iv_i & \lambda\sum_{i=1}^3r_iv_i & 1-\lambda\sum_{i=1}^3s_iv_i
         \end{array}
       \right).
$$
So corank $\Hess d_a$ is 3 if and only if $$1-\lambda\sum_{i=1}^3a_iv_i=1-\lambda\sum_{i=1}^3d_iv_i=1-\lambda\sum_{i=1}^3s_iv_i=0,$$ $$\sum_{i=1}^3b_iv_i=\sum_{i=1}^3c_iv_i=\sum_{i=1}^3r_iv_i=0.$$ This is equivalent to $$\sum_{i=1}^3(a_i-d_i)v_i=\sum_{i=1}^3(-a_i-d_i+2s_i)v_i=\sum_{i=1}^3b_iv_i=\sum_{i=1}^3c_iv_i=\sum_{i=1}^3r_iv_i=0,$$ which is equivalent to $$\langle v,B_j\rangle=0$$ for $j=1,\ldots,5$. This means that $v$ is in $E_p^{\perp}$.

Furthermore, we get $$\frac{1}{\lambda}=\frac{1}{3}\sum_{i=1}^3(a_i+d_i+s_i)v_i=\langle H,v\rangle,$$ so $\lambda=\frac{1}{\langle H,v\rangle}=\frac{1}{\kappa_u^6}.$
\end{proof}

\begin{coro}
Let $p\in M^3_{\reg}$ be a point in a generic 3-manifold and consider $q\in M^2_{\reg}\subset \R^5$ given by a generic normal section. There exists a unique umbilical focus at $p$ if and only if there exists a unique umbilical focus at $q=i_1^{-1}(p)$. Moreover, the radii of the umbilical focal hyperspheres are the same.
\end{coro}
\begin{proof}
By Proposition \ref{contact3var}, for 3-manifolds we have a unique umbilical focus if and only if the curvature locus is a planar region and $Aff_p$ does not contain $p$. Taking a generic normal section we obtain a surface in $\R^5$ whose curvature ellipse lies in the plane $Aff_p$. Since the section is generic, the ellipse will be non-degenerate (taking a non generic normal section could make $q$ a semiumbilic point). Therefore $Aff_p=Aff_q$ and $q$ is an $\M_3$-point. From \cite{Costa/Fuster/Moraes} we know that at $\M_3$-points (and only at these points) there is a unique umbilical focus. The statement about the radii follows from Proposition \ref{sameku}.
\end{proof}



\begin{thebibliography}{22}

\bibitem{BenediniOset} {\sc P. Benedini Riul and R. Oset Sinha} {\it A relation between the curvature ellipse and the curvature parabola}.  Adv. Geom. 19 (2019), no. 3, 389--399.

\bibitem{BenediniOset2} {\sc P. Benedini Riul and R. Oset Sinha} {\it Relating second order geometry of manifolds through projections and normal sections}. To appear in Publicacions Matem\`{a}tiques. arXiv:1909.07307

\bibitem{Benedini/Sinha/Ruas} {\sc P. Benedini Riul, R. Oset Sinha and M. A. S. Ruas} {\it The geometry of corank $1$ surfaces in $\mathbb{R}^{4}$}. Q. J. Math. 70 (2019), no. 3, 767--795.

\bibitem{BenediniRuasSacramento} {\sc P. Benedini Riul, M. A. S. Ruas and A. de Jesus Sacramento} {\it Singular 3-manifolds in $\mathbb{R}^{5}$}. Preprint. arXiv:1911.00360

\bibitem{Binotto/Costa/Fuster} { \sc R. R. Binotto, S. I. Costa and M. C. Romero Fuster} {\it The curvature Veronese of a 3-manifold in Euclidean space}. Real and complex singularities: Amer. Math. Soc., Providence, RI (2016), (Contemp. Math., v. 675), p. 25--44.

\bibitem{BruceNogueira} {\sc{J. W. Bruce and A. C. Nogueira}} {\it Surfaces in $\mathbb{R}^{4}$ and duality}. Quart. J. Math. Oxford Ser. 49 (1998), 433--443.



\bibitem{BruceTari} {\sc{J. W. Bruce and F. Tari}} {\it Families of surfaces in $\mathbb{R}^{4}$}. Proc. Edinb. Math. Soc. (2) 45 (2002), no. 1, 181--203.


\bibitem{Costa/Fuster/Moraes} {\sc{S. I. R. Costa, M. S. Moraes and M. C. Romero Fuster}} {\it Geometric contact of surfaces immersed in $\mathbb{R}^{n}$, $n\geqslant5$}. Differential Geom, Appl. 27 (2009), 442--454.



\bibitem{Deolindo} {\sc{J. L. Deolindo-Silva}} {\it Cross-ratio invariants for surfaces in 4-space}. Bull Braz Math Soc, New Series (2020). https://doi.org/10.1007/s00574-020-00221-w.

\bibitem{dreibelbis} {\sc{D. Dreibelbis}} {\it Self-conjugate vectors of immersed 3-manifolds in $\mathbb R^6$}.  Topology Appl. 159 (2012), no. 2, 450--456.


\bibitem{GarciaMochidaFusterRuas} {\sc{R. Garcia, D. K. H. Mochida, M. C. Romero Fuster and M. A. S. Ruas}} {\it Inflection points and topology of surfaces in 4-space}. Trans. Amer. Math. Soc. 352 (2000), 3029--3043.



\bibitem{Livro} {\sc{S. Izumiya, M. C. Romero Fuster, M. A. S. Ruas and F. Tari}} {\it Differential Geometry from Singularity Theory Viewpoint}. World Scientific Publishing Co. Pte. Ltd., Hackensack, NJ, 2016. xiii+368 pp. ISBN: 978-981-4590-44-0

\bibitem{Rieger} {\sc{C. Klotz, O. Pop and J. H. Rieger}} {\it Real double-points of deformations of $\mathcal{A}$-simple map-germs from $\mathbb{R}^{n}$ to
$\mathbb{R}^{2n}$}. Math. Proc. Camb. Phil. Soc. (2007), 142-341.


\bibitem{Little} {\sc{J. A. Little}} {\it On singularities of submanifolds of higher dimensional Euclidean
spaces}. Ann. Mat. Pura Appl. 83 (4) (1969), 261--335.

\bibitem {MartinsBallesteros} {\sc{L. F. Martins and J. J. Nu\~{n}o-Ballesteros,}} {\it{Contact properties of surfaces in $\mathbb{R}^{3}$ with corank $1$ singularities}}. Tohoku Math. J. 67 (2015), 105--124.





\bibitem{MochidaFusterRuas3}{\sc{D. K. H. Mochida, M. C. Romero Fuster and M. A. S. Ruas,}} {\it Inflection points and nonsingular embeddings of surfaces in $\mathbb R^5$}, Rocky Mountain J. Math. 33 (2003) 995--1010.

\bibitem{MochidaFusterRuas} {\sc{D. K. H. Mochida, M. C. Romero Fuster and M. A. S. Ruas,}} {\it The geometry of surfaces in $4$-space from a contact viewpoint}. Geom. Dedicata 54 (1995), 323--332.

\bibitem{MochidaFusterRuas2} {\sc{D. K. H. Mochida, M. C. Romero Fuster and M. A. S. Ruas,}} {\it Osculating
hyperplanes and asymptotic directions of codimension two submanifolds of Euclidean spaces}. Geom. Dedicata 77 (1999), 305--315.




\bibitem{BallesterosTari} {\sc{J. J. Nu\~{n}o-Ballesteros and F. Tari,}} {\it Surfaces in $\mathbb{R}^{4}$ and their projections to $3$-spaces}. Proc. Roy. Soc. Edinburgh Sect. A, 137 (2007), 1313--1328.

\bibitem{Oset/Saji} {\sc{R. Oset Sinha and K. Saji,}} {\it On the geometry of folded cuspidal edges}.   Rev. Mat. Complut. 31 (2018), no. 3, 627--650.

\bibitem{OsetSinhaTari} {\sc{R. Oset Sinha and F. Tari,}} {\it Projections of surfaces in $\mathbb{R}^{4}$ to $\mathbb{R}^{3}$ and the geometry of their singular
   images}. Rev. Mat. Iberoam. 32 (2015), no. 1, 33--50.



\bibitem{Fuster/Ruas/Tari} {\sc{M. C. Romero Fuster,  M. A. S. Ruas and F. Tari,}} {\it Asymptotic curves on surfaces in $\mathbb{R}^{5}$}. Communications in Contemporary Maths. 10 (2008), 1--27.

\bibitem{SUY} {\sc{K. Saji, M. Umehara, and K. Yamada,}} {\it The geometry of fronts}. Ann. of Math (2) 169 (2009), 491--529.



\end{thebibliography}
\end{document}